\documentclass[12pt,reqno]{amsart}
\usepackage{amsthm, amscd, amsfonts, amssymb, graphicx, color}
\usepackage[bookmarksnumbered, colorlinks, plainpages,linkcolor=green,anchorcolor=black,citecolor=blue,urlcolor=black]{hyperref}
\usepackage{enumerate}
\usepackage{exscale}
\usepackage[all]{xy}
\usepackage{relsize}
\usepackage{pdfsync}
\numberwithin{equation}{section}

\makeatletter
\@namedef{subjclassname@2020}{%
  \textup{2020} Mathematics Subject Classification}
\makeatother

\usepackage[includemp,body={398pt,550pt},footskip=30pt,%
            marginparwidth=60pt,marginparsep=10pt]{geometry}

\newtheorem{theorem}{Theorem}[section]
\newtheorem{lemma}[theorem]{Lemma}
\newtheorem{proposition}[theorem]{Proposition}
\newtheorem{corollary}[theorem]{Corollary}
\theoremstyle{definition}

\numberwithin{equation}{section}

\newcommand{\h}{\mathcal{H}}

\newcommand{\C}{\mathbb{C}}

\newcommand{\Z}{\mathbb{Z}}

\newcommand{\res}{\mathrm{restricted}}

\begin{document}
\title[Reverse Carleson inequalities for weighted Fock spaces]{Reverse Carleson inequalities for weighted Fock spaces with $A_{\infty}$-type weights}
\date{May 30, 2026.
}
\author[J. Chen]{Jiale Chen}
\address{Jiale Chen, School of Mathematics and Statistics, Shaanxi Normal University, Xi'an 710119, China.}
\email{jialechen@snnu.edu.cn}

\thanks{This work was supported by National Natural Science Foundation (No. 12501170), Young Talent Fund of Association for Science and Technology in Shaanxi (No. 20260506) and the Fundamental Research Funds for the Central Universities (No. GK202601003) of China.}



\begin{abstract}
  \noindent We characterize dominating sets for weighted Fock spaces $F^p_{\alpha,w}$ induced by Muckenhoupt $A_{\infty}$ weights, which leads to a sufficient condition for reverse Carleson measures for the spaces $F^p_{\alpha,w}$ with restricted $A_{\infty}$ weights. Applications to invertibility of Toeplitz operators, closed range property of Volterra operators, and atomic decomposition of $F^p_{\alpha,w}$ are given.
\end{abstract}
\maketitle


\section{Introduction}
\allowdisplaybreaks[4]

The purpose of this paper is to investigate dominating sets and reverse Carleson measures for a class of weighted Fock spaces. We begin with some basic notions. A non-negative and locally integrable function $w$ on the complex plane $\C$ is said to be a weight. Given $0<p,\alpha<\infty$ and a weight $w$, the space $L^p_{\alpha,w}$ consists of measurable functions $f$ on $\C$ such that
$$\|f\|^p_{L^p_{\alpha,w}}:=\int_{\C}|f(z)|^pe^{-\frac{p\alpha}{2}|z|^2}w(z)dA(z)<\infty,$$
where $dA$ is the Lebesgue measure on $\C$. The weighted Fock space $F^p_{\alpha,w}$ is defined as
$$F^p_{\alpha,w}:=L^p_{\alpha,w}\cap\h(\C)$$
with the inherited (quasi-)norm, where $\h(\C)$ is the space of entire functions on $\C$. If $w=\frac{\alpha}{\pi}$, then we obtain the weighted spaces $L^p_{\alpha}$ and the standard Fock spaces $F^p_{\alpha}$. We refer to \cite{Zh} for a brief account on Fock spaces.

It is well-known that for $0<\alpha<\infty$, the Fock space $F^2_{\alpha}$ is a closed subspace of $L^2_{\alpha}$. Consequently, there exists an orthogonal projection $P_{\alpha}$ from $L^2_{\alpha}$ onto $F^2_{\alpha}$, which is called the Fock projection and is given by
$$P_{\alpha}(f)(z):=\frac{\alpha}{\pi}\int_{\C}f(\xi)\overline{K_z(\xi)}e^{-\alpha|\xi|^2}dA(\xi),\quad z\in\C, \quad f\in L^2_{\alpha},$$
where $K_z(\xi):=e^{\alpha \overline{z}\xi}$ is the reproducing kernel of $F^2_{\alpha}$. For any $1\leq p<\infty$, the operator $P_{\alpha}$ actually extends to a bounded projection from $L^p_{\alpha}$ onto $F^p_{\alpha}$; see \cite[Theorem 7.1]{JPR}.

To characterize the invertible products of Toeplitz operators on Fock spaces $F^p_{\alpha}$, Isralowitz \cite{Is} investigated the weighted boundedness of the Fock projection $P_{\alpha}$ on $L^p_{\alpha,w}$. Here and in the sequel, we use $Q$ to denote a square in $\C$ with sides parallel to the coordinate axes, and write $\ell(Q)$ for its side length. As usual, $p'$ denotes the conjugate exponent of $p$, i.e. $1/p+1/p'=1$, for $1<p<\infty$. Given $1<p<\infty$, a weight $w$ is said to belong to the class $A^{\res}_p$ if for some (equivalently for any) fixed $r>0$,
$$\sup_{Q:\ell(Q)=r}\left(\frac{1}{A(Q)}\int_{Q}wdA\right)
\left(\frac{1}{A(Q)}\int_{Q}w^{-\frac{p'}{p}}dA\right)^{\frac{p}{p'}}<\infty,$$
and $w$ is said to belong to the class $A^{\res}_1$ if for some (equivalently for any) fixed $r>0$,
$$\sup_{Q:\ell(Q)=r}\left(\frac{1}{A(Q)}\int_{Q}wdA\right)\left\|w^{-1}\right\|_{L^{\infty}(Q)}<\infty.$$
It was proved in \cite[Theorem 3.1]{Is} for $1<p<\infty$ and later in \cite[Proposition 2.7]{CFP} for $p=1$ that $P_{\alpha}$ is bounded on $L^p_{\alpha,w}$ if and only if $w\in A_p^{\res}$. A major drawback of $A^{\res}_p$ weights is that they do not satisfy a reverse H\"older inequality (see \cite[Remark 2.5]{CFP}), which is a key feature of the classical Muckenhoupt weights. Recall that for $1<p<\infty$, the Muckenhoupt $A_p$-class consists of weights $w$ with
$$\sup_{Q}\left(\frac{1}{A(Q)}\int_{Q}wdA\right)
\left(\frac{1}{A(Q)}\int_{Q}w^{-\frac{p'}{p}}dA\right)^{\frac{p}{p'}}<\infty,$$
and the Muckenhoupt $A_1$-class consists of weights $w$ with
$$\sup_{Q}\left(\frac{1}{A(Q)}\int_{Q}wdA\right)\left\|w^{-1}\right\|_{L^{\infty}(Q)}<\infty.$$
It is clear that for $1\leq p<\infty$, $A_p\subsetneq A_p^{\res}$. As usual, we write
$$A_{\infty}:=\bigcup_{1\leq p<\infty}A_p$$
and
$$A^{\res}_{\infty}:=\bigcup_{1\leq p<\infty}A^{\res}_p.$$
Recently, the function and operator theory on weighted Fock spaces induced by weights from $A^{\res}_{\infty}$ has attracted some attention; see \cite{CFP23,CFP,Ch24,Ch25-1,Ch24-1,Ch25,CHW,Xu}.

In this paper, we are going to investigate dominating sets and reverse Carleson measures for weighted Fock spaces $F^p_{\alpha,w}$ induced by $A_{\infty}$-type weights. A positive Borel measure $\mu$ on $\C$ is said to be a reverse Carleson measure for $F^p_{\alpha,w}$ if there exists $C>0$ such that for any $f\in F^p_{\alpha,w}$,
$$\|f\|^p_{F^p_{\alpha,w}}\leq C\int_{\C}|f(z)|^pe^{-\frac{p\alpha}{2}|z|^2}d\mu(z),$$
and a measurable set $G\subset\C$ is a dominating set for $F^p_{\alpha,w}$ if $\chi_GwdA$ is a reverse Carleson measure for $F^p_{\alpha,w}$. Here and in the sequel, for a set $E\subset\C$, $\chi_E$ is the characteristic function of $E$. The investigation of dominating sets and reverse Carleson measures for spaces of analytic functions was initiated by Luecking in the setting of Bergman spaces; see \cite{Lu81,Lu84,Lu85,Lu85-1}. We refer to \cite{KR19,THLA23,TLHA21,TS} for recent developments on weighed Bergman spaces. In the setting of Fock spaces, Wang and Zhao \cite{WZ} characterized dominating sets for the Fock space $F^2_{1/2}$ and applied it to determine the invertibility of Toeplitz operators on $F^2_{1/2}$. It was proved in \cite[Theorem 1.2]{WZ} that a measurable set $G\subset\C$ is a dominating set for $F^2_{1/2}$ if and only if there exists $R>0$ and $\delta\in(0,1)$ such that for all $a\in\C$,
$$A\big(G\cap D(a,R)\big)>\delta\pi R^2.$$
Here and in the sequel, $D(a,R)$ denotes the Euclidean disk centered at $a$ with the radius $R$. In this paper, we first generalize the above characterization to weighted Fock spaces $F^p_{\alpha,w}$ induced by $A_{\infty}$-type weights, and then establish a sufficient condition for reverse Carleson measures.

Our first result is the following characterization of dominating sets for weighted Fock spaces $F^p_{\alpha,w}$ induced by Muckenhoupt $A_{\infty}$ weights. For a measurable subset $E\subset\C$, we write $w(E):=\int_EwdA$.

\begin{theorem}\label{main1}
Let $0<p,\alpha<\infty$, $w\in A_{\infty}$, and let $G\subset\C$ be a measurable set. Then $G$ is a dominating set for $F^p_{\alpha,w}$ if and only if there exist $R>0$ and $\delta\in(0,1)$ such that for all $a\in\C$,
$$w\big(G\cap D(a,R)\big)>\delta w\big(D(a,R)\big).$$
\end{theorem}

The reverse H\"older property of Muckenhoupt $A_{\infty}$ weights plays an important role in Theorem \ref{main1}. When $w\in A^{\res}_{\infty}$, we can characterize dominating sets for the weighted Fock spaces $F^p_{\alpha,\widehat{w}}$ (see Theorem \ref{suff}), where the weight $\widehat{w}$ is defined by
$$\widehat{w}(z)=\frac{w\big(D(z,1)\big)}{A\big(D(z,1)\big)},\quad z\in\C.$$
This characterization together with the reverse H\"older property of Muckenhoupt $A_{\infty}$ weights leads to the proof of Theorem \ref{main1}.

Our second result establishes a sufficient condition for reverse Carleson measures for weighted Fock spaces $F^p_{\alpha,w}$ with $w\in A^{\res}_{\infty}$. Here, for a positive Borel measure $\mu$ on $\C$, the function $\widehat{\mu}_w$ is defined by
$$ \quad\widehat{\mu}_w(z):=\frac{\mu\big(D(z,1)\big)}{w\big(D(z,1)\big)},\quad z\in\C.$$

\begin{theorem}\label{main2}
Let $0<p,\alpha,R,\epsilon<\infty$, $\delta\in(0,1)$, and let $w\in A^{\res}_{\infty}$. Then there exists an $r_{\ref{main2}}=r_{\ref{main2}}(p,\alpha,R,\epsilon,\delta,w)>0$ such that each $r\in(0,r_{\ref{main2}})$ satisfies the following property: whenever $\mu$ is a positive Borel measure on $\C$ for which $\|\widehat{\mu}_w\|_{L^{\infty}}<\infty$ and for which the set
$$G:=\left\{z\in\C:\frac{\mu\big(D(z,r)\big)}{\widehat{w}\big(D(z,r)\big)}>\epsilon \|\widehat{\mu}_w\|_{L^{\infty}}\right\}$$
satisfies
$$\widehat{w}\big(G\cap D(a,R)\big)>\delta\widehat{w}\big(D(a,R)\big)$$
for all $a\in\C$, $\mu$ is a reverse Carleson measure for $F^p_{\alpha,w}$.
\end{theorem}

The proof of Theorem \ref{main2} is based on the characterization of dominating sets for the spaces $F^p_{\alpha,\widehat{w}}$ and a difference quotient estimate for functions in weighted Fock spaces.

The rest part of the paper is organized as follows. In Section \ref{pre}, some preliminary results are given. Theorems \ref{main1} and \ref{main2} are respectively proved in Sections \ref{proof1} and \ref{proof2}. We finally give some applications of our main results in Section \ref{app}, including a characterization of invertible Toeplitz operators on $F^2_{\alpha,w}$, a characterization of closed range Volterra operators on $F^p_{\alpha,w}$, and an atomic decomposition for $F^p_{\alpha,w}$.

\section{Preliminaries}\label{pre}

In this section, we give some preliminary results that will be used in the sequel. The following sub-mean inequality can be found in \cite[Lemma 2.32]{Zh}.

\begin{lemma}\label{sub-mean1}
Let $0<p,\alpha,r<\infty$. Then for any $f\in\h(\C)$ and $a\in\C$,
$$|f(a)|^pe^{-\frac{p\alpha}{2}|a|^2}\leq\frac{p\alpha}{2\pi\left(1-e^{-{p\alpha r^2}/{2}}\right)}\int_{D(a,r)}|f(z)|^pe^{-\frac{p\alpha}{2}|z|^2}dA(z).$$
\end{lemma}

The following Lipschitz type estimate for entire functions can be found in \cite[Lemma 4.6]{Zh}.

\begin{lemma}\label{lip}
Let $0<p,\alpha,R<\infty$. Then there exists $C_{\ref{lip}}=C_{\ref{lip}}(p,\alpha,R)>0$ such that for any $f\in\h(\C)$ and any $a,z\in\C$ with $|a-z|<R/3$,
$$\left||f(z)|e^{-\frac{\alpha}{2}|z|^2}-|f(a)|e^{-\frac{\alpha}{2}|a|^2}\right|^p
\leq C_{\ref{lip}}|z-a|^p\int_{D(a,R)}|f(u)|^pe^{-\frac{p\alpha}{2}|u|^2}dA(u).$$
\end{lemma}

We now recall some estimates involving $A_{\infty}$-type weights. The following lemma was proved in \cite[Lemma 3.4]{Is}. Here and in the sequel, for each $r>0$, we treat $r\Z^2$ as a subset of $\C$ in the natural way. For $a\in\C$ and $r>0$, $Q_r(a)$ denotes the square centered at $a$ with side length $\ell(Q)=r$.

\begin{lemma}\label{w-sq}
Let $w\in A^{\res}_{\infty}$ and $r>0$. Then there exists $C_{\ref{w-sq}}=C_{\ref{w-sq}}(w,r)>0$ such that for any $\nu,\nu'\in r\mathbb{Z}^2$,
$$\frac{w\big(Q_r(\nu)\big)}{w\big(Q_r(\nu')\big)}\leq C_{\ref{w-sq}}^{|\nu-\nu'|}.$$
\end{lemma}

Based on the above lemma, one can directly obtain the following estimate; see \cite[Remark 2.3]{CFP}.

\begin{lemma}\label{w-esti}
Let $w\in A^{\res}_{\infty}$ and $R>0$. Then there exists $C_{\ref{w-esti}}=C_{\ref{w-esti}}(R,w)\geq1$ such that for any $a,z\in\C$ with $|a-z|<R$,
$$C_{\ref{w-esti}}^{-1}w\big(D(a,1)\big)\leq w\big(D(z,R)\big)\leq C_{\ref{w-esti}}w\big(D(a,1)\big).$$
\end{lemma}

Recall that for $a\in\C$, the reproducing kernel $K_a$ is given by $K_a(z)=e^{\alpha\overline{a}z}$. The following norm estimate was proved in \cite[Proposition 4.1]{CFP}.

\begin{lemma}\label{test}
Let $0<p,\alpha<\infty$ and $w\in A^{\res}_{\infty}$. Then there exists $C_{\ref{test}}=C_{\ref{test}}(p,\alpha,w)\geq1$ such that for any $a\in\C$,
$$C_{\ref{test}}^{-1}e^{\frac{\alpha}{2}|a|^2}w\big(D(a,1)\big)^{1/p}\leq\|K_a\|_{F^p_{\alpha,w}}
\leq C_{\ref{test}}e^{\frac{\alpha}{2}|a|^2}w\big(D(a,1)\big)^{1/p}.$$
\end{lemma}

The following lemma establishes an equivalent norm for the spaces $F^p_{\alpha,w}$ via the weight $\widehat{w}$; see \cite[Lemma 3.2]{CHW}.

\begin{lemma}\label{hat-norm}
Let $0<p,\alpha<\infty$ and $w\in A^{\res}_{\infty}$. Then $F^p_{\alpha,w}=F^p_{\alpha,\widehat{w}}$ with equivalent norms.
\end{lemma}

As stated in the Introduction, the Muckenhoupt $A_{\infty}$ weights satisfy the following reverse H\"older inequality; see \cite[Theorem 7.2.2]{Gr} or \cite[Theorem 2.3]{HPR}.

\begin{theorem}\label{rh}
Let $w\in A_{\infty}$. Then there exist $s=s(w)>1$ and $C_{\ref{rh}}=C_{\ref{rh}}(w)>0$ such that for any $a\in\C$ and $r>0$,
$$\left(\frac{1}{\pi r^2}\int_{D(a,r)}w^sdA\right)^{1/s}\leq \frac{C_{\ref{rh}}}{\pi r^2}\int_{D(a,r)}wdA.$$
\end{theorem}

\section{Proof of Theorem \ref{main1}}\label{proof1}

In this section, we are going to prove Theorem \ref{main1}.

\subsection{Necessity}

The necessity of Theorem \ref{main1} is based on the following lemma.

\begin{lemma}\label{limit0}
For any $\beta>0$ and $w\in A^{\res}_{\infty}$,
$$\lim_{r\to\infty}\sup_{a\in\C}\frac{1}{w\big(D(a,1)\big)}\int_{\C\setminus D(a,r)}e^{-\beta|z-a|^2}w(z)dA(z)=0.$$
\end{lemma}
\begin{proof}
Fix $a\in\C$ and suppose $r>\sqrt{2}$. We claim that
$$\C\setminus D(0,r)\subset\bigcup_{\nu\in\mathbb{Z}^2\setminus D(0,r/2)}Q_1(\nu).$$
In fact, if $|z|\geq r$ and $z\in Q_1(\nu)$ for some $\nu\in\mathbb{Z}^2$, then
$$|\nu|\geq|z|-|z-\nu|\geq r-\frac{\sqrt{2}}{2}>\frac{r}{2}.$$
Consequently, we have
\begin{align*}
I(a):&=\frac{1}{w\big(D(a,1)\big)}\int_{\C\setminus D(a,r)}e^{-\beta|z-a|^2}w(z)dA(z)\\
&=\frac{1}{w\big(D(a,1)\big)}\int_{\C\setminus D(0,r)}e^{-\beta|z|^2}w(z+a)dA(z)\\
&\leq\frac{1}{w\big(D(a,1)\big)}\sum_{\nu\in\mathbb{Z}^2\setminus D(0,r/2)}\int_{Q_1(\nu)}e^{-\beta|z|^2}w(z+a)dA(z).
\end{align*}
Noting that for $z\in Q_1(\nu)$, $|z|^2\geq\frac{1}{2}|\nu|^2-\frac{1}{2}$, we obtain
\begin{align*}
I(a)&\leq\frac{e^{\beta/2}}{w\big(D(a,1)\big)}\sum_{\nu\in\mathbb{Z}^2\setminus D(0,r/2)}e^{-\frac{\beta}{2}|\nu|^2}
    \int_{Q_1(\nu)}w(z+a)dA(z)\\
&\leq\frac{e^{\beta/2}}{w\big(Q_1(a)\big)}\sum_{\nu\in\mathbb{Z}^2\setminus D(0,r/2)}e^{-\frac{\beta}{2}|\nu|^2}w\big(Q_1(\nu+a)\big).
\end{align*}
By Lemma \ref{w-sq}, there exists $C_{\ref{w-sq}}=C_{\ref{w-sq}}(w)>0$ such that
$$I(a)\leq e^{\beta/2}\sum_{\nu\in\mathbb{Z}^2\setminus D(0,r/2)}e^{-\frac{\beta}{2}|\nu|^2}C_{\ref{w-sq}}^{|\nu|}.$$
Since $a\in\C$ is arbitrary, the desired result follows.
\end{proof}

We are now ready to establish the necessity of Theorem \ref{main1}, which is contained in the following proposition.

\begin{proposition}\label{necessity}
Let $0<p,\alpha<\infty$, $w\in A^{\res}_{\infty}$, and let $G\subset\C$ be a measurable set. If $G$ is a dominating set for $F^p_{\alpha,w}$, then there exist $R>0$ and $\delta\in(0,1)$ such that for all $a\in\C$,
$$w\big(G\cap D(a,R)\big)>\delta w\big(D(a,R)\big).$$
\end{proposition}
\begin{proof}
Suppose that $G$ is a dominating set for $F^p_{\alpha,w}$. Then there exists $C>0$ such that for any $f\in F^p_{\alpha,w}$,
\begin{equation}\label{nece}
\|f\|^p_{F^p_{\alpha,w}}\leq C\int_{G}|f(z)|^pe^{-\frac{p\alpha}{2}|z|^2}w(z)dA(z).
\end{equation}
For any $a\in\C$, let
$$f_a(z):=w\big(D(a,1)\big)^{-1/p}e^{\alpha\overline{a}z-\frac{\alpha}{2}|a|^2},\quad z\in\C.$$
Then by Lemma \ref{test},
$$C_{\ref{test}}^{-1}\leq\|f_a\|_{F^p_{\alpha,w}}\leq C_{\ref{test}}.$$
Applying \eqref{nece} to $f_a$, we obtain that
$$\int_{G}|f_a(z)|^pe^{-\frac{p\alpha}{2}|z|^2}w(z)dA(z)\geq\frac{1}{CC_{\ref{test}}^p}.$$
On the other hand, we can use Lemma \ref{limit0} to find $R>1$ such that
\begin{align*}
\int_{G\setminus D(a,R)}&|f_a(z)|^pe^{-\frac{p\alpha}{2}|z|^2}w(z)dA(z)\\
&\leq\frac{1}{w\big(D(a,1)\big)}\int_{\C\setminus D(a,R)}e^{-\frac{p\alpha}{2}|z-a|^2}w(z)dA(z)<\frac{1}{2CC_{\ref{test}}^p}.
\end{align*} 
Consequently,
\begin{align*}
\frac{w\big(G\cap D(a,R)\big)}{w\big(D(a,1)\big)}
&\geq\frac{1}{w\big(D(a,1)\big)}\int_{G\cap D(a,R)}e^{-\frac{p\alpha}{2}|z-a|^2}w(z)dA(z)\\
&=\int_{G\cap D(a,R)}|f_a(z)|^pe^{-\frac{p\alpha}{2}|z|^2}w(z)dA(z)\\
&=\left(\int_{G}-\int_{G\setminus D(a,R)}\right)|f_a(z)|^pe^{-\frac{p\alpha}{2}|z|^2}w(z)dA(z)\\
&>\frac{1}{2CC_{\ref{test}}^p},
\end{align*}
which, combined with Lemma \ref{w-esti}, implies that
$$\frac{w\big(G\cap D(a,R)\big)}{w\big(D(a,R)\big)}
>\frac{1}{2CC_{\ref{test}}^p}\frac{w\big(D(a,1)\big)}{w\big(D(a,R)\big)}\geq\frac{1}{2CC_{\ref{w-esti}}C_{\ref{test}}^p}.$$
The arbitrariness of $a\in\C$ finishes the proof.
\end{proof}

\subsection{Sufficiency}

We now turn to the sufficiency of Theorem \ref{main1}, which is based on the following characterization of dominating sets for the spaces $F^p_{\alpha,\widehat{w}}$.

\begin{theorem}\label{suff}
Let $0<p,\alpha<\infty$, $w\in A^{\res}_{\infty}$, and let $G\subset\C$ be a measurable set. Then $G$ is a dominating set for $F^p_{\alpha,\widehat{w}}$ if and only if there exist $R>0$ and $\delta\in(0,1)$ such that for any $a\in\C$,
\begin{equation}\label{hat-suff}
\widehat{w}\big(G\cap D(a,R)\big)>\delta \widehat{w}\big(D(a,R)\big).
\end{equation}
\end{theorem}

Thanks to Theorem \ref{suff}, we can establish the sufficiency of Theorem \ref{main1}.

\begin{proof}[Proof of Theorem \ref{main1}. Sufficiency]
Suppose that there exist $R>0$ and $\delta\in(0,1)$ such that for any $a\in\C$,
\begin{equation}\label{condition}
w\big(G\cap D(a,R)\big)>\delta w\big(D(a,R)\big).
\end{equation}
We are going to show that $G$ is a dominating set for $F^p_{\alpha,w}$. Define
$$\Omega:=\left\{z\in\C:w(z)\geq\frac{\delta}{2R^2C_{\ref{w-esti}}}\widehat{w}(z)\right\},$$
where $C_{\ref{w-esti}}$ is the constant from Lemma \ref{w-esti} associated with $R$ and $w$. Then we apply Lemma \ref{w-esti} to obtain that for any $a\in\C$,
\begin{align*}
w\big(D(a,R)\setminus\Omega\big)
&=\int_{D(a,R)\setminus\Omega}w(z)dA(z)\\
&<\frac{\delta}{2R^2C_{\ref{w-esti}}}\int_{D(a,R)\setminus\Omega}\widehat{w}(z)dA(z)\\
&\leq\frac{\delta w\big(D(a,R)\big)}{2\pi R^2}A\big(D(a,R)\setminus\Omega\big)\\
&\leq\frac{\delta}{2}w\big(D(a,R)\big),
\end{align*}
which implies that
$$w\big(\Omega\cap D(a,R)\big)=w\big(D(a,R)\big)-w\big(D(a,R)\setminus\Omega\big)>\left(1-\frac{\delta}{2}\right)w\big(D(a,R)\big).$$
Consequently,
$$w\big(G\cap\Omega\cap D(a,R)\big)>\left(1-\frac{\delta}{2}\right)w\big(D(a,R)\big)-w\big(\big(\Omega\cap D(a,R)\big)\setminus G\big).$$
Noting that the condition \eqref{condition} implies that
$$w\big(G\cap \Omega\cap D(a,R)\big)>\delta w\big(D(a,R)\big)-w\big(\big(G\cap D(a,R)\big)\setminus\Omega\big),$$
we obtain that
\begin{equation}\label{delta/4}
w\big(G\cap \Omega\cap D(a,R)\big)>\frac{\delta}{4}w\big(D(a,R)\big).
\end{equation}
Since $w\in A_{\infty}$, we apply Lemma \ref{w-esti}, H\"older's inequality and Theorem \ref{rh} to obtain that for some $s=s(w)>1$,
\begin{align*}
&w\big(G\cap\Omega\cap D(a,R)\big)\\
&\ =\int_{G\cap\Omega\cap D(a,R)}\frac{w(z)}{\widehat{w}(z)^{1/s'}}\widehat{w}(z)^{1/s'}dA(z)\\
&\ \leq\left(\frac{\pi C_{\ref{w-esti}}}{w\big(D(a,R)\big)}\right)^{1/s'}\int_{G\cap\Omega\cap D(a,R)}w(z)\widehat{w}(z)^{1/s'}dA(z)\\
&\ \leq\left(\frac{\pi C_{\ref{w-esti}}}{w\big(D(a,R)\big)}\right)^{1/s'}\left(\int_{D(a,R)}w^sdA\right)^{1/s}
    \cdot\widehat{w}\big(G\cap\Omega\cap D(a,R)\big)^{1/s'}\\
&\ \leq\frac{C_{\ref{w-esti}}^{1/s'}C_{\ref{rh}}}{R^{2/s'}}w\big(D(a,R)\big)^{1/s}
    \cdot\widehat{w}\big(G\cap\Omega\cap D(a,R)\big)^{1/s'}.
\end{align*}
Combining this with \eqref{delta/4} yields that
\begin{align*}
\widehat{w}\big(G\cap\Omega\cap D(a,R)\big)&>\frac{\delta^{s'}R^2}{4^{s'}C_{\ref{w-esti}}C_{\ref{rh}}^{s'}}w\big(D(a,R)\big)\\
&=\frac{\delta^{s'}}{4^{s'}C_{\ref{w-esti}}C_{\ref{rh}}^{s'}}\int_{D(a,R)}\frac{w\big(D(a,R)\big)}{\pi}dA(z)\\
&\geq\frac{\delta^{s'}}{4^{s'}C^2_{\ref{w-esti}}C_{\ref{rh}}^{s'}}\widehat{w}\big(D(a,R)\big),
\end{align*}
where the last inequality is due to Lemma \ref{w-esti}. Therefore, using Theorem \ref{suff}, we obtain that $G\cap\Omega$ is a dominating set for the space $F^p_{\alpha,\widehat{w}}$. Consequently, there exists $C>0$ such that for any $f\in F^p_{\alpha,\widehat{w}}$,
\begin{align*}
\|f\|^p_{F^p_{\alpha,\widehat{w}}}&\leq C\int_{G\cap\Omega}|f(z)|^pe^{-\frac{p\alpha}{2}|z|^2}\widehat{w}(z)dA(z)\\
&\leq\frac{2R^2CC_{\ref{w-esti}}}{\delta}\int_{G}|f(z)|^pe^{-\frac{p\alpha}{2}|z|^2}w(z)dA(z),
\end{align*}
which, in conjunction with Lemma \ref{hat-norm}, gives that $G$ is a dominating set for $F^p_{\alpha,w}$.
\end{proof}

We now turn to the proof of Theorem \ref{suff}. Note that by Lemma \ref{w-esti}, $w\in A^{\res}_{\infty}$ implies that $\widehat{w}\in A^{\res}_{\infty}$. Hence the necessity of Theorem \ref{suff} is a direct consequence of Proposition \ref{necessity}. The rest part of this section is devoted to proving the sufficiency of Theorem \ref{suff}. Fix $R>0$ and $\delta\in(0,1)$ such that \eqref{hat-suff} holds for any $a\in\C$. 
Write
$$K_0=K_0(p,\alpha,R):=\frac{p\alpha R^2}{2\left(1-e^{-p\alpha R^2/2}\right)}$$
and
$$\gamma_0=\gamma_0(p,\alpha,R,w):=e^{\frac{p\alpha}{4}R^2+C_{\ref{w-esti}}^2}K_0,$$
where $C_{\ref{w-esti}}$ is the constant from Lemma \ref{w-esti} associated with $R$ and $w$. It is easy to verify that $K_0>1$. For each $f\in\h(\C)$, $a\in\C$ and $0<\lambda<\min\left\{1,2^{-p}K_0\right\}$, define
$$E_{\lambda}(a)=E_{\lambda}(f,a):
=\left\{z\in D(a,R):\gamma_0|f(z)|^pe^{-\frac{p\alpha}{2}|z|^2}>\lambda|f(a)|^pe^{-\frac{p\alpha}{2}|a|^2}\right\}$$
and
$$B_{\lambda}f(a):=\frac{\gamma_0}{\widehat{w}\big(E_{\lambda}(a)\big)}
\int_{E_{\lambda}(a)}|f(z)|^pe^{-\frac{p\alpha}{2}|z|^2}\widehat{w}(z)dA(z).$$
Noting that for any $z\in D(a,R)$, Lemma \ref{w-esti} gives that
\begin{equation}\label{hat-in}
\widehat{w}(z)\geq\frac{w\big(D(a,R)\big)}{\pi C_{\ref{w-esti}}}
=\frac{1}{\pi R^2C_{\ref{w-esti}}}\int_{D(a,R)}\frac{w\big(D(a,R)\big)}{\pi}dA(u)
\geq\frac{\widehat{w}\big(D(a,R)\big)}{\pi R^2C_{\ref{w-esti}}^2},
\end{equation}
we apply Lemma \ref{sub-mean1} to obtain that
\begin{align*}
B_{\lambda}f(a)
&\geq\frac{\gamma_0}{\widehat{w}\big(D(a,R)\big)}\int_{D(a,R)}|f(z)|^pe^{-\frac{p\alpha}{2}|z|^2}\widehat{w}(z)dA(z)\\
&>\frac{C_{\ref{w-esti}}^2K_0}{\widehat{w}\big(D(a,R)\big)}\int_{D(a,R)}|f(z)|^pe^{-\frac{p\alpha}{2}|z|^2}\widehat{w}(z)dA(z)\\
&\geq\frac{K_0}{\pi R^2}\int_{D(a,R)}|f(z)|^pe^{-\frac{p\alpha}{2}|z|^2}dA(z)\\
&\geq|f(a)|^pe^{-\frac{p\alpha}{2}|a|^2}.
\end{align*}

We will need the following sub-mean inequality.

\begin{lemma}\label{sub-mean2}
For any $f\in\h(\C)$ and $a\in\C$,
$$\log\left(|f(a)|^pe^{-\frac{p\alpha}{2}|a|^2}\right)
<\frac{1}{\widehat{w}\big(D(a,R)\big)}\int_{D(a,R)}\log\left(\gamma_0|f(z)|^pe^{-\frac{p\alpha}{2}|z|^2}\right)\widehat{w}(z)dA(z).$$
\end{lemma}
\begin{proof}
Using the same method as in the proof of \cite[Lemma 3.2]{WZ}, we can establish
$$\log\left(|f(a)|^pe^{-\frac{p\alpha}{2}|a|^2}\right)<
\frac{1}{\pi R^2}\int_{D(a,R)}\log\left(e^{\frac{p\alpha}{4}R^2}K_0|f(z)|^pe^{-\frac{p\alpha}{2}|z|^2}\right)dA(z),$$
which, combined with \eqref{hat-in}, implies that
\begin{align*}
&\log\left(|f(a)|^pe^{-\frac{p\alpha}{2}|a|^2}\right)\\
&\ <\frac{C_{\ref{w-esti}}^2}{\widehat{w}\big(D(a,R)\big)}\int_{D(a,R)}
    \log\left(e^{\frac{p\alpha}{4}R^2}K_0|f(z)|^pe^{-\frac{p\alpha}{2}|z|^2}\right)\widehat{w}(z)dA(z)\\
&\ =\frac{1}{\widehat{w}\big(D(a,R)\big)}\int_{D(a,R)}\log\left(\gamma_0|f(z)|^pe^{-\frac{p\alpha}{2}|z|^2}\right)\widehat{w}(z)dA(z).
\end{align*}
The proof is complete.
\end{proof}

The following lemma gives an estimation of the weighted measure for the set $E_{\lambda}(a)$.

\begin{lemma}\label{m-ratio}
For any $f\in\h(\C)$, $a\in\C$ and $0<\lambda<\min\left\{1,2^{-p}K_0\right\}$,
$$\frac{\widehat{w}\big(E_{\lambda}(a)\big)}{\widehat{w}\big(D(a,R)\big)}>
\frac{\log\frac{1}{\lambda}}{\log\frac{B_{\lambda}f(a)}{|f(a)|^pe^{-\frac{p\alpha}{2}|a|^2}}+\log\frac{1}{\lambda}}.$$
\end{lemma}
\begin{proof}
By Lemma \ref{sub-mean2}, we have
\begin{align*}
&\log\left(|f(a)|^pe^{-\frac{p\alpha}{2}|a|^2}\right)\\
&\ <\frac{\widehat{w}\big(E_{\lambda}(a)\big)}{\widehat{w}\big(D(a,R)\big)}\frac{1}{\widehat{w}\big(E_{\lambda}(a)\big)}
    \int_{E_{\lambda}(a)}\log\left(\gamma_0|f(z)|^pe^{-\frac{p\alpha}{2}|z|^2}\right)\widehat{w}(z)dA(z)\\
&\ \quad+\frac{1}{\widehat{w}\big(D(a,R)\big)}\int_{D(a,R)\setminus E_{\lambda}(a)}
    \log\left(\gamma_0|f(z)|^pe^{-\frac{p\alpha}{2}|z|^2}\right)\widehat{w}(z)dA(z).
\end{align*}
Using Jensen's inequality and the definition of $E_{\lambda}(a)$ yields that
\begin{align*}
&\log\left(|f(a)|^pe^{-\frac{p\alpha}{2}|a|^2}\right)\\
&\ <\frac{\widehat{w}\big(E_{\lambda}(a)\big)}{\widehat{w}\big(D(a,R)\big)}
    \log\left(\frac{\gamma_0}{\widehat{w}\big(E_{\lambda}(a)\big)}\int_{E_{\lambda}(a)}|f(z)|^pe^{-\frac{p\alpha}{2}|z|^2}\widehat{w}(z)dA(z)\right)\\
&\ \quad+\frac{1}{\widehat{w}\big(D(a,R)\big)}\int_{D(a,R)\setminus E_{\lambda}(a)}
    \log\left(\lambda|f(a)|^pe^{-\frac{p\alpha}{2}|a|^2}\right)\widehat{w}(z)dA(z)\\
&\ =\frac{\widehat{w}\big(E_{\lambda}(a)\big)}{\widehat{w}\big(D(a,R)\big)}\log B_{\lambda}f(a)\\
&\ \quad+\left(1-\frac{\widehat{w}\big(E_{\lambda}(a)\big)}{\widehat{w}\big(D(a,R)\big)}\right)
    \left(\log \lambda+\log\left(|f(a)|^pe^{-\frac{p\alpha}{2}|a|^2}\right)\right),
\end{align*}
which, combined with the fact that $B_{\lambda}f(a)>|f(a)|^pe^{-\frac{p\alpha}{2}|a|^2}$ and $0<\lambda<1$, gives the desired result.
\end{proof}

For $\epsilon\in(0,1)$, let the sets $A_{\epsilon}$ and $B_{\epsilon}$ are defined respectively by
$$A_{\epsilon}:=\left\{a\in\C:|f(a)|^pe^{-\frac{p\alpha}{2}|a|^2}
\leq\frac{\epsilon}{\widehat{w}\big(D(a,R)\big)}\int_{D(a,R)}|f(z)|^pe^{-\frac{p\alpha}{2}|z|^2}\widehat{w}(z)dA(z)\right\}$$
and
$$B_{\epsilon}:=\left\{a\in\C:|f(a)|^pe^{-\frac{p\alpha}{2}|a|^2}\leq\epsilon^{1+\frac{2}{p}}B_{\lambda}f(a)\right\}.$$
We have the following estimates.

\begin{lemma}\label{A-inte}
For any $f\in F^p_{\alpha,\widehat{w}}$,
$$\int_{A_{\epsilon}}|f(z)|^pe^{-\frac{p\alpha}{2}|z|^2}\widehat{w}(z)dA(z)\leq
C_{\ref{w-esti}}^2\epsilon\|f\|_{F^p_{\alpha,\widehat{w}}}^p.$$
\end{lemma}
\begin{proof}
By the definition of $A_{\epsilon}$ and Fubini's theorem, we have
\begin{align*}
&\int_{A_{\epsilon}}|f(z)|^pe^{-\frac{p\alpha}{2}|z|^2}\widehat{w}(z)dA(z)\\
&\ \leq\epsilon\int_{A_{\epsilon}}\frac{\widehat{w}(z)}{\widehat{w}\big(D(z,R)\big)}
    \int_{D(z,R)}|f(u)|^pe^{-\frac{p\alpha}{2}|u|^2}\widehat{w}(u)dA(u)dA(z)\\
&\ =\epsilon\int_{\C}|f(u)|^pe^{-\frac{p\alpha}{2}|u|^2}\widehat{w}(u)
    \int_{A_{\epsilon}\cap D(u,R)}\frac{\widehat{w}(z)}{\widehat{w}\big(D(z,R)\big)}dA(z)dA(u).
\end{align*}
By Lemma \ref{w-esti}, we know that for any $z\in\C$,
\begin{equation}\label{hatD}
\widehat{w}\big(D(z,R)\big)=\int_{D(z,R)}\widehat{w}(\xi)dA(\xi)\geq\frac{R^2w\big(D(z,R)\big)}{C_{\ref{w-esti}}}
\geq\frac{\pi R^2\widehat{w}(z)}{C_{\ref{w-esti}}^2},
\end{equation}
which finishes the proof.
\end{proof}

\begin{lemma}\label{B-inte}
There exist $C_{\ref{B-inte}}=C_{\ref{B-inte}}(p,\alpha,R,w)>0$ and $\epsilon_0=\epsilon_0(p,\alpha,R,w)\in(0,1)$ such that for any $f\in F^p_{\alpha,\widehat{w}}$ and $\epsilon\in(0,\epsilon_0)$,
$$\int_{B_{\epsilon}}|f(z)|^pe^{-\frac{p\alpha}{2}|z|^2}\widehat{w}(z)dA(z)\leq C_{\ref{B-inte}}\epsilon\|f\|_{F^p_{\alpha,\widehat{w}}}^p.$$
\end{lemma}
\begin{proof}
By Lemma \ref{A-inte}, it suffices to estimate the integral over $B_{\epsilon}\setminus A_{\epsilon}$. By the definition of $B_{\epsilon}$ and Fubini's theorem, for any $f\in F^p_{\alpha,\widehat{w}}$,
\begin{align*}
&\int_{B_{\epsilon}\setminus A_{\epsilon}}|f(z)|^pe^{-\frac{p\alpha}{2}|z|^2}\widehat{w}(z)dA(z)\\
&\ \leq\epsilon^{1+\frac{2}{p}}\int_{B_{\epsilon}\setminus A_{\epsilon}}B_{\lambda}f(z)\widehat{w}(z)dA(z)\\
&\ =\gamma_0\epsilon^{1+\frac{2}{p}}\int_{B_{\epsilon}\setminus A_{\epsilon}}\frac{\widehat{w}(z)}{\widehat{w}\big(E_{\lambda}(z)\big)}
    \int_{E_{\lambda}(z)}|f(u)|^pe^{-\frac{p\alpha}{2}|u|^2}\widehat{w}(u)dA(u)dA(z)\\
&\ =\gamma_0\epsilon^{1+\frac{2}{p}}\int_{\C}|f(u)|^pe^{-\frac{p\alpha}{2}|u|^2}\widehat{w}(u)
    \int_{B_{\epsilon}\setminus A_{\epsilon}}\frac{\chi_{E_{\lambda}(z)}(u)}{\widehat{w}\big(E_{\lambda}(z)\big)}\widehat{w}(z)dA(z)dA(u)\\
&\ \leq\gamma_0\epsilon^{1+\frac{2}{p}}\int_{\C}|f(u)|^pe^{-\frac{p\alpha}{2}|u|^2}\widehat{w}(u)
    \int_{B_{\epsilon}\setminus A_{\epsilon}}\frac{\chi_{D(z,R)}(u)}{\widehat{w}\big(E_{\lambda}(z)\big)}\widehat{w}(z)dA(z)dA(u)\\
&\ =\gamma_0\epsilon^{1+\frac{2}{p}}\int_{\C}|f(u)|^pe^{-\frac{p\alpha}{2}|u|^2}\widehat{w}(u)
    \int_{(B_{\epsilon}\setminus A_{\epsilon})\cap D(u,R)}\frac{\widehat{w}(z)}{\widehat{w}\big(E_{\lambda}(z)\big)}dA(z)dA(u).
\end{align*}
For any $z\in\C$, Lemma \ref{w-esti} indicates that
$$\widehat{w}\big(E_{\lambda}(z)\big)=\int_{E_{\lambda}(z)}\widehat{w}(\xi)dA(\xi)
\geq\frac{w\big(D(z,R)\big)}{\pi C_{\ref{w-esti}}}A\big(E_{\lambda}(z)\big)
\geq\frac{\widehat{w}(z)}{C_{\ref{w-esti}}^2}A\big(E_{\lambda}(z)\big).$$
Therefore,
\begin{align}\label{B-A}
&\int_{B_{\epsilon}\setminus A_{\epsilon}}|f(z)|^pe^{-\frac{p\alpha}{2}|z|^2}\widehat{w}(z)dA(z)\nonumber\\
&\qquad\leq C_{\ref{w-esti}}^2\gamma_0\epsilon^{1+\frac{2}{p}}\int_{\C}|f(u)|^pe^{-\frac{p\alpha}{2}|u|^2}\widehat{w}(u)
\int_{(B_{\epsilon}\setminus A_{\epsilon})\cap D(u,R)}\frac{dA(z)}{A\big(E_{\lambda}(z)\big)}dA(u).
\end{align}
It remains to estimate the measure $A\big(E_{\lambda}(z)\big)$ for $z\in B_{\epsilon}\setminus A_{\epsilon}$. Fix $z\notin A_{\epsilon}$. Then
\begin{equation}\label{outA}
|f(z)|e^{-\frac{\alpha}{2}|z|^2}\geq\epsilon^{1/p}\left(\widehat{|f|^p_w}(z)\right)^{1/p},
\end{equation}
where
$$\widehat{|f|^p_w}(z):=\frac{1}{\widehat{w}\big(D(z,R)\big)}\int_{D(z,R)}|f(\xi)|^pe^{-\frac{p\alpha}{2}|\xi|^2}\widehat{w}(\xi)dA(\xi).$$
By Lemma \ref{lip}, there exists $C_{\ref{lip}}=C_{\ref{lip}}(p,\alpha,R)>0$ such that for any $\zeta\in D(z,R/3)$,
\begin{align*}
&\left||f(\zeta)|e^{-\frac{\alpha}{2}|\zeta|^2}-|f(z)|e^{-\frac{\alpha}{2}|z|^2}\right|^p\\
&\ \leq C_{\ref{lip}}|\zeta-z|^p\int_{D(z,R)}|f(\xi)|^pe^{-\frac{p\alpha}{2}|\xi|^2}dA(\xi)\\
&\ \leq \pi R^2C_{\ref{lip}}C_{\ref{w-esti}}^2\frac{|\zeta-z|^p}{\widehat{w}\big(D(z,R)\big)}
    \int_{D(z,R)}|f(\xi)|^pe^{-\frac{p\alpha}{2}|\xi|^2}\widehat{w}(\xi)dA(\xi)\\
&\ =\pi R^2C_{\ref{lip}}C_{\ref{w-esti}}^2|\zeta-z|^p\widehat{|f|^p_w}(z),
\end{align*}
where in the second inequality we have used \eqref{hat-in}. Choose
$$\epsilon_0:=\min\left\{\frac{1}{2},\left(\frac{2R}{3}\right)^p\cdot\pi R^2C_{\ref{lip}}C_{\ref{w-esti}}^2\right\}$$
and fix $\epsilon\in(0,\epsilon_0)$. Then for any $\zeta\in\C$ with
$$|\zeta-z|<\left(\frac{\epsilon}{2^p\pi R^2C_{\ref{lip}}C_{\ref{w-esti}}^2}\right)^{1/p},$$
we have $\zeta\in D(z,R/3)$, and consequently,
\begin{align*}
\left||f(\zeta)|e^{-\frac{\alpha}{2}|\zeta|^2}-|f(z)|e^{-\frac{\alpha}{2}|z|^2}\right|
&\leq\left(\pi R^2C_{\ref{lip}}C_{\ref{w-esti}}^2\right)^{1/p}|\zeta-z|\left(\widehat{|f|^p_w}(z)\right)^{1/p}\\
&<\frac{\epsilon^{1/p}}{2}\left(\widehat{|f|^p_w}(z)\right)^{1/p},
\end{align*}
which together with \eqref{outA} implies that
\begin{align*}
|f(\zeta)|e^{-\frac{\alpha}{2}|\zeta|^2}
&\geq|f(z)|e^{-\frac{\alpha}{2}|z|^2}-\left||f(\zeta)|e^{-\frac{\alpha}{2}|\zeta|^2}-|f(z)|e^{-\frac{\alpha}{2}|z|^2}\right|\\
&>|f(z)|e^{-\frac{\alpha}{2}|z|^2}-\frac{\epsilon^{1/p}}{2}\left(\widehat{|f|^p_w}(z)\right)^{1/p}\\
&\geq\frac{1}{2}|f(z)|e^{-\frac{\alpha}{2}|z|^2}.
\end{align*}
Therefore,
$$\gamma_0|f(\zeta)|^pe^{-\frac{p\alpha}{2}|\zeta|^2}>2^{-p}K_0|f(z)|^pe^{-\frac{p\alpha}{2}|z|^2}
>\lambda|f(z)|^pe^{-\frac{p\alpha}{2}|z|^2}.$$
That is, $\zeta\in E_{\lambda}(z)$. Hence we obtain that
$$D\left(z,\left(\frac{\epsilon}{2^p\pi R^2C_{\ref{lip}}C_{\ref{w-esti}}^2}\right)^{1/p}\right)\subset E_{\lambda}(z),$$
which gives
$$A\big(E_{\lambda}(z)\big)\geq\pi\left(\frac{\epsilon}{2^p\pi R^2C_{\ref{lip}}C_{\ref{w-esti}}^2}\right)^{2/p}.$$
Inserting this into \eqref{B-A} establishes
$$\int_{B_{\epsilon}\setminus A_{\epsilon}}|f(z)|^pe^{-\frac{p\alpha}{2}|z|^2}\widehat{w}(z)dA(z)
\leq4\pi^{\frac{2}{p}}\gamma_0R^{2+\frac{4}{p}}C_{\ref{lip}}^{\frac{2}{p}}C_{\ref{w-esti}}^{2+\frac{4}{p}}
  \epsilon\|f\|_{F^p_{\alpha,\widehat{w}}}^p.$$
The proof is complete.
\end{proof}

We are now ready to prove the sufficiency of Theorem \ref{suff}.

\begin{proof}[Proof of Theorem \ref{suff}. Sufficiency]
Fix $\epsilon=\frac{1}{2}\min\left\{C_{\ref{B-inte}}^{-1},\epsilon_0\right\}$, where $\epsilon_0=\epsilon_0(p,\alpha,R,w)\in(0,1)$ is from Lemma \ref{B-inte}. Then for any $f\in F^p_{\alpha,\widehat{w}}$, Lemma \ref{B-inte} yields that
$$\int_{B_{\epsilon}}|f(z)|^pe^{-\frac{p\alpha}{2}|z|^2}\widehat{w}(z)dA(z)\leq\frac{1}{2}\|f\|_{F^p_{\alpha,\widehat{w}}}^p.$$
Consequently,
\begin{equation}\label{C-B-control}
\|f\|_{F^p_{\alpha,\widehat{w}}}^p\leq2\int_{\C\setminus B_{\epsilon}}|f(z)|^pe^{-\frac{p\alpha}{2}|z|^2}\widehat{w}(z)dA(z).
\end{equation}
For any $z\in\C\setminus B_{\epsilon}$, Lemma \ref{m-ratio} gives that
\begin{align*}
\frac{\widehat{w}\big(E_{\lambda}(z)\big)}{\widehat{w}\big(D(z,R)\big)}
&>\frac{\log\frac{1}{\lambda}}{\log\frac{B_{\lambda}f(z)}{|f(z)|^pe^{-\frac{p\alpha}{2}|z|^2}}+\log\frac{1}{\lambda}}\\
&>\frac{\log\frac{1}{\lambda}}{(1+2/p)\log\frac{1}{\epsilon}+\log\frac{1}{\lambda}},
\end{align*}
so we may choose $\lambda\in\big(0,\min\{1,2^{-p}K_0\}\big)$ small enough such that
$$\frac{\widehat{w}\big(E_{\lambda}(z)\big)}{\widehat{w}\big(D(z,R)\big)}\geq1-\frac{\delta}{2}.$$
This together with \eqref{hat-suff} implies that for $z\in\C\setminus B_{\epsilon}$,
\begin{align*}
\widehat{w}\big(G\cap E_{\lambda}(z)\big)
&=\widehat{w}\big(G\cap D(z,R)\big)-\widehat{w}\big(G\cap(D(a,R)\setminus E_{\lambda}(z))\big)\\
&\geq\widehat{w}\big(G\cap D(z,R)\big)-\widehat{w}\big(D(z,R)\setminus E_{\lambda}(z)\big)\\
&\geq\frac{\delta}{2}\widehat{w}\big(D(z,R)\big).
\end{align*}
Therefore, we use the definition of $E_{\lambda}(z)$, Fubini's theorem and \eqref{hatD} to establish that for $f\in F^p_{\alpha,\widehat{w}}$,
\begin{align*}
&\int_{\C\setminus B_{\epsilon}}|f(z)|^pe^{-\frac{p\alpha}{2}|z|^2}\widehat{w}(z)dA(z)\\
&\ \leq\frac{2}{\delta}\int_{\C\setminus B_{\epsilon}}|f(z)|^pe^{-\frac{p\alpha}{2}|z|^2}
    \frac{\widehat{w}\big(G\cap E_{\lambda}(z)\big)}{\widehat{w}\big(D(z,R)\big)}\widehat{w}(z)dA(z)\\
&\ =\frac{2}{\delta}\int_{\C\setminus B_{\epsilon}}|f(z)|^pe^{-\frac{p\alpha}{2}|z|^2}\frac{\widehat{w}(z)}{\widehat{w}\big(D(z,R)\big)}
    \int_{G\cap E_{\lambda}(z)}\widehat{w}(\xi)dA(\xi)dA(z)\\
&\ \leq\frac{2\gamma_0}{\lambda\delta}\int_{\C\setminus B_{\epsilon}}\frac{\widehat{w}(z)}{\widehat{w}\big(D(z,R)\big)}
    \int_{G\cap E_{\lambda}(z)}|f(\xi)|^pe^{-\frac{p\alpha}{2}|\xi|^2}\widehat{w}(\xi)dA(\xi)dA(z)\\
&\ \leq\frac{2\gamma_0}{\lambda\delta}\int_{\C\setminus B_{\epsilon}}\frac{\widehat{w}(z)}{\widehat{w}\big(D(z,R)\big)}
    \int_{G\cap D(z,R)}|f(\xi)|^pe^{-\frac{p\alpha}{2}|\xi|^2}\widehat{w}(\xi)dA(\xi)dA(z)\\
&\ =\frac{2\gamma_0}{\lambda\delta}\int_{G}|f(\xi)|^pe^{-\frac{p\alpha}{2}|\xi|^2}\widehat{w}(\xi)
    \int_{(\C\setminus B_{\epsilon})\cap D(\xi,R)}\frac{\widehat{w}(z)}{\widehat{w}\big(D(z,R)\big)}dA(z)dA(\xi)\\
&\ \leq\frac{2\gamma_0C_{\ref{w-esti}}^2}{\lambda\delta}\int_{G}|f(\xi)|^pe^{-\frac{p\alpha}{2}|\xi|^2}\widehat{w}(\xi)dA(\xi).
\end{align*}
Combining the above estimate with \eqref{C-B-control} yields that for any $f\in F^p_{\alpha,\widehat{w}}$,
$$\|f\|_{F^p_{\alpha,\widehat{w}}}^p
\leq\frac{4\gamma_0C_{\ref{w-esti}}^2}{\lambda\delta}\int_{G}|f(\xi)|^pe^{-\frac{p\alpha}{2}|\xi|^2}\widehat{w}(\xi)dA(\xi).$$
Hence we conclude that $G$ is a dominating set for $F^p_{\alpha,\widehat{w}}$ and the proof is complete.
\end{proof}

\section{Proof of Theorem \ref{main2}}\label{proof2}

The purpose of this section is to prove Theorem \ref{main2}. To this end, we establish the following difference quotient estimate for functions in weighted Fock spaces.

\begin{lemma}\label{difference}
Let $0<p,\alpha<\infty$ and $w\in A^{\res}_{\infty}$. Suppose that $\mu$ is a positive Borel measure on $\C$ such that $\left\|\widehat{\mu}_{w}\right\|_{L^{\infty}}<\infty$. Then there exists $C_{\ref{difference}}=C_{\ref{difference}}(p,\alpha,w)>0$ such that for any $f\in F^p_{\alpha,\widehat{w}}$ and any $r\in(0,1/3)$,
\begin{align*}
\int_{\C}\int_{\C}\frac{\chi_{D(z,r)}(\xi)}{\widehat{w}\big(D(\xi,r)\big)}
&\left||f(\xi)|e^{-\frac{\alpha}{2}|\xi|^2}-|f(z)|e^{-\frac{\alpha}{2}|z|^2}\right|^p\widehat{w}(z)dA(z)d\mu(\xi)\\
&\leq C_{\ref{difference}}\left\|\widehat{\mu}_w\right\|_{L^{\infty}}
    r^p\|f\|_{F^p_{\alpha,\widehat{w}}}^p.
\end{align*}
\end{lemma}
\begin{proof}
Fix $f\in F^p_{\alpha,\widehat{w}}$. Applying Lemmas \ref{lip} and \ref{w-esti} with $R=1$, we obtain that there exist $C_{\ref{lip}}=C_{\ref{lip}}(p,\alpha)>0$ and $C_{\ref{w-esti}}=C_{\ref{w-esti}}(w)>0$ such that for any $z,\xi\in\C$ with $|z-\xi|<r<1/3$,
\begin{align*}
&\left||f(\xi)|e^{-\frac{\alpha}{2}|\xi|^2}-|f(z)|e^{-\frac{\alpha}{2}|z|^2}\right|^p\\
&\ \leq C_{\ref{lip}}|z-\xi|^p\int_{D(\xi,1)}|f(u)|^pe^{-\frac{p\alpha}{2}|u|^2}dA(u)\\
&\ \leq \pi C_{\ref{lip}}C_{\ref{w-esti}}\frac{r^p}{w\big(D(\xi,1)\big)}
    \int_{D(\xi,1)}|f(u)|^pe^{-\frac{p\alpha}{2}|u|^2}\widehat{w}(u)dA(u).
\end{align*}
Multiply this by ${\chi_{D(z,r)}(\xi)}/{\widehat{w}\big(D(\xi,r)\big)}$ and integrate with respect to $\widehat{w}(z)dA(z)$ to establish
\begin{align*}
&\int_{\C}\frac{\chi_{D(z,r)}(\xi)}{\widehat{w}\big(D(\xi,r)\big)}
\left||f(\xi)|e^{-\frac{\alpha}{2}|\xi|^2}-|f(z)|e^{-\frac{\alpha}{2}|z|^2}\right|^p\widehat{w}(z)dA(z)\\
&\ \leq\pi C_{\ref{lip}}C_{\ref{w-esti}}\frac{r^p}{w\big(D(\xi,1)\big)}
\int_{D(\xi,1)}|f(u)|^pe^{-\frac{p\alpha}{2}|u|^2}\widehat{w}(u)dA(u),
\end{align*}
which, combined with Fubini's theorem and Lemma \ref{w-esti}, gives that
\begin{align*}
&\int_{\C}\int_{\C}\frac{\chi_{D(z,r)}(\xi)}{\widehat{w}\big(D(\xi,r)\big)}
    \left||f(\xi)|e^{-\frac{\alpha}{2}|\xi|^2}-|f(z)|e^{-\frac{\alpha}{2}|z|^2}\right|^p\widehat{w}(z)dA(z)d\mu(\xi)\\
&\ \leq\pi C_{\ref{lip}}C_{\ref{w-esti}}r^p\int_{\C}\frac{1}{w\big(D(\xi,1)\big)}
    \int_{D(\xi,1)}|f(u)|^pe^{-\frac{p\alpha}{2}|u|^2}\widehat{w}(u)dA(u)d\mu(\xi)\\
&\ =\pi C_{\ref{lip}}C_{\ref{w-esti}}
    r^p\int_{\C}|f(u)|^pe^{-\frac{p\alpha}{2}|u|^2}\widehat{w}(u)
    \int_{D(u,1)}\frac{d\mu(\xi)}{w\big(D(\xi,1)\big)}dA(u)\\
&\ \leq\pi C_{\ref{lip}}C_{\ref{w-esti}}^2
    r^p\int_{\C}|f(u)|^pe^{-\frac{p\alpha}{2}|u|^2}\widehat{w}(u)
    \frac{\mu\big(D(u,1)\big)}{w\big(D(u,1)\big)}dA(u)\\
&\ \leq\pi C_{\ref{lip}}C_{\ref{w-esti}}^2\left\|\widehat{\mu}_w\right\|_{L^{\infty}}
    r^p\|f\|_{F^p_{\alpha,\widehat{w}}}^p.
\end{align*}
The proof is complete.
\end{proof}

We are now ready to prove Theorem \ref{main2}.

\begin{proof}[Proof of Theorem \ref{main2}]
In the case $p\geq1$, we use Lemma \ref{difference}, the triangle inequality and Fubini's theorem to obtain that for any $f\in F^p_{\alpha,\widehat{w}}$ and $r\in(0,1/3)$,
\begin{align*}
&\left(C_{\ref{difference}}\left\|\widehat{\mu}_w\right\|_{L^{\infty}}r^p\|f\|_{F^p_{\alpha,\widehat{w}}}^p\right)^{1/p}\\
&\ \geq\left(\int_{\C}\int_{\C}\frac{\chi_{D(z,r)}(\xi)}{\widehat{w}\big(D(\xi,r)\big)}
    \left||f(\xi)|e^{-\frac{\alpha}{2}|\xi|^2}-|f(z)|e^{-\frac{\alpha}{2}|z|^2}\right|^p\widehat{w}(z)dA(z)d\mu(\xi)\right)^{1/p}\\
&\ \geq\left(\int_{\C}\int_{\C}\frac{\chi_{D(z,r)}(\xi)}{\widehat{w}\big(D(\xi,r)\big)}
    |f(z)|^pe^{-\frac{p\alpha}{2}|z|^2}\widehat{w}(z)dA(z)d\mu(\xi)\right)^{1/p}\\
&\ \qquad\qquad -\left(\int_{\C}\int_{\C}\frac{\chi_{D(z,r)}(\xi)}{\widehat{w}\big(D(\xi,r)\big)}
    |f(\xi)|^pe^{-\frac{p\alpha}{2}|\xi|^2}\widehat{w}(z)dA(z)d\mu(\xi)\right)^{1/p}\\
&\ =\left(\int_{\C}|f(z)|^pe^{-\frac{p\alpha}{2}|z|^2}\widehat{w}(z)
    \int_{D(z,r)}\frac{d\mu(\xi)}{\widehat{w}\big(D(\xi,r)\big)}dA(z)\right)^{1/p}\\
&\ \qquad\qquad-\left(\int_{\C}|f(\xi)|^pe^{-\frac{p\alpha}{2}|\xi|^2}d\mu(\xi)\right)^{1/p}.
\end{align*}
Note that for $\xi,\zeta\in D(z,r)$ and $u\in D(\xi,r)$, we have $|u-\zeta|<3r<1$, and so Lemma \ref{w-esti} with $R=1$ gives that there exists $C_{\ref{w-esti}}=C_{\ref{w-esti}}(w)>0$ such that for any $\xi\in D(z,r)$,
\begin{align*}
\widehat{w}\big(D(\xi,r)\big)&=\frac{1}{\pi r^2}\int_{D(z,r)}\int_{D(\xi,r)}\widehat{w}(u)dA(u)dA(\zeta)\\
&\leq\frac{C_{\ref{w-esti}}}{\pi r^2}\int_{D(z,r)}\int_{D(\xi,r)}\widehat{w}(\zeta)dA(u)dA(\zeta)\\
&=C_{\ref{w-esti}}\widehat{w}\big(D(z,r)\big).
\end{align*}
Therefore,
\begin{align*}
&\left(C_{\ref{difference}}\left\|\widehat{\mu}_w\right\|_{L^{\infty}}r^p\|f\|_{F^p_{\alpha,\widehat{w}}}^p\right)^{1/p}\\
&\ \geq\left(\frac{1}{C_{\ref{w-esti}}}\int_{\C}|f(z)|^pe^{-\frac{p\alpha}{2}|z|^2}\widehat{w}(z)
    \frac{\mu\big(D(z,r)\big)}{\widehat{w}\big(D(z,r)\big)}dA(z)\right)^{1/p}\\
&\ \qquad\qquad-\left(\int_{\C}|f(\xi)|^pe^{-\frac{p\alpha}{2}|\xi|^2}d\mu(\xi)\right)^{1/p}\\
&\ >\left(\frac{\epsilon\|\widehat{\mu}_{w}\|_{L^{\infty}}}{C_{\ref{w-esti}}}
    \int_{G}|f(z)|^pe^{-\frac{p\alpha}{2}|z|^2}\widehat{w}(z)dA(z)\right)^{1/p}\\
&\ \qquad\qquad-\left(\int_{\C}|f(\xi)|^pe^{-\frac{p\alpha}{2}|\xi|^2}d\mu(\xi)\right)^{1/p}.
\end{align*}
By Theorem \ref{suff}, there exists $C=C(p,\alpha,R,\delta,w)>0$ such that
$$\int_{G}|f(z)|^pe^{-\frac{p\alpha}{2}|z|^2}\widehat{w}(z)dA(z)\geq C\|f\|_{F^p_{\alpha,\widehat{w}}}^p.$$
We finally conclude that
\begin{align*}
&\left(C_{\ref{difference}}\left\|\widehat{\mu}_w\right\|_{L^{\infty}}r^p\|f\|_{F^p_{\alpha,\widehat{w}}}^p\right)^{1/p}\\
&\quad \geq\left(\frac{C\epsilon\|\widehat{\mu}_{w}\|_{L^{\infty}}}{C_{\ref{w-esti}}}\|f\|_{F^p_{\alpha,\widehat{w}}}^p\right)^{1/p}
-\left(\int_{\C}|f(\xi)|^pe^{-\frac{p\alpha}{2}|\xi|^2}d\mu(\xi)\right)^{1/p}.
\end{align*}
Choosing $r_{\ref{main2}}=\min\left\{\frac{1}{3},\left(\frac{C\epsilon}{C_{\ref{w-esti}}C_{\ref{difference}}}\right)^{1/p}\right\}$, then we obtain that whenever $r\in(0,r_{\ref{main2}})$,
$$\left(\int_{\C}|f(\xi)|^pe^{-\frac{p\alpha}{2}|\xi|^2}d\mu(\xi)\right)^{1/p}
\geq\left(\left(\frac{C\epsilon}{C_{\ref{w-esti}}}\right)^{1/p}-C_{\ref{difference}}^{1/p}r\right)
\left\|\widehat{\mu}_w\right\|_{L^{\infty}}^{1/p}\|f\|_{F^p_{\alpha,\widehat{w}}}.$$
This together with Lemma \ref{hat-norm} yields that $\mu$ is a reverse Carleson measure for $F^p_{\alpha,w}$. In the case $0<p<1$ we use the inequality
$$|a-b|^p\geq|a|^p-|b|^p,\quad a,b\in\C$$
instead of the triangle inequality and the result follows from exactly the same method.
\end{proof}

\section{Applications}\label{app}

In this section, we give some applications of our main results.

\subsection{Invertibility of Toeplitz operators}

Given $\varphi\in L^{\infty}(\C)$, the Toeplitz operator $T_{\varphi}$ is defined by
$$T_{\varphi}f:=P(\varphi f),\quad f\in F^2_{\alpha,w},$$
where $P$ is the orthogonal projection from $L^2_{\alpha,w}$ onto $F^2_{\alpha,w}$. Recently, the author \cite{Ch25-1} characterized the bounded, compact and Schatten class Toeplitz operators with positive measure symbols on the weighted Fock spaces $F^2_{\alpha,w}$ with $w\in A^{\res}_{\infty}$. Here we can use Theorem \ref{main1} to characterize the invertibility of Toeplitz operators $T_{\varphi}$ with positive bounded symbols acting on $F^2_{\alpha,w}$ induced by Muckenhoupt $A_{\infty}$ weights. The proof is the same as that of \cite[Corollary 3]{Lu81} and so we omit it; see also \cite[Theorem 1.1]{WZ}.

\begin{corollary}
Let $\alpha>0$ and $w\in A_{\infty}$, and let $\varphi$ be a non-negative function in $L^{\infty}(\C)$. Then the following statements are equivalent:
\begin{enumerate}
	\item [(a)] $T_{\varphi}$ is invertible on $F^2_{\alpha,w}$;
	\item [(b)] there exists $C>0$ such that for any $f\in F^2_{\alpha,w}$,
	$$\int_{\C}|f(z)|^2\varphi(z)^2e^{-\alpha|z|^2}w(z)dA(z)\geq C\|f\|_{F^2_{\alpha,w}}^2;$$
	\item [(c)] there exist $R,\epsilon>0$ and $\delta\in(0,1)$ such that the set $G_\epsilon=\{z\in\C:\varphi(z)\geq \epsilon\}$ satisfies
	$$w\big(G_{\epsilon}\cap D(a,R)\big)>\delta w\big(D(a,R)\big),\quad \forall a\in\C.$$
\end{enumerate}
\end{corollary}

\subsection{Closed range Volterra operators}

Given $g\in\h(\C)$, the Volterra type integration operator $J_g$ is defined by
$$J_gf(z):=\int_0^zf(\xi)g'(\xi)d\xi,\quad f\in\h(\C).$$
It was proved in \cite[Theorem 3.1]{Xu} that for $0<p,\alpha<\infty$ and $w\in A^{\res}_{\infty}$, $J_g$ is bounded on $F^p_{\alpha,w}$ if and only if $g(z)=c_2z^2+c_1z+c_0$ for some $c_0,c_1,c_2\in\C$. We here determine when $J_g$ has closed range on $F^p_{\alpha,w}$. To this end, we need the following direct consequence of Theorem \ref{suff}. The proof is omitted.

\begin{corollary}\label{multiplier}
Let $\varphi$ be a non-negative function in $L^{\infty}(\C)$ and write $G_{\epsilon}:=\{z\in\C:\varphi(z)\geq \epsilon\}$ for $\epsilon>0$. Let $0<p,\alpha<\infty$ and $w\in A^{\res}_{\infty}$. Then there exists $C>0$ such that
$$\int_{\C}|f(z)|^p\varphi(z)e^{-\frac{p\alpha}{2}|z|^2}\widehat{w}(z)dA(z)\geq C\|f\|_{F^p_{\alpha,\widehat{w}}}^p$$
holds for all $f\in F^p_{\alpha,\widehat{w}}$ if and only if there exist $R,\epsilon>0$ and $\delta\in(0,1)$ such that
\begin{equation}\label{G-t}
	\widehat{w}\big(G_{\epsilon}\cap D(a,R)\big)>\delta\widehat{w}\big(D(a,R)\big),\quad \forall a\in\C.
\end{equation}
\end{corollary}

We are now ready to characterize the non-trivial Volterra operators $J_g$ that have closed range on $F^p_{\alpha,w}$.

\begin{theorem}
Let $0<p,\alpha<\infty$ and $w\in A^{\res}_{\infty}$. Suppose that $g(z)=c_2z^2+c_1z+c_0$ for some $c_0,c_1,c_2\in\C$ with $|c_1|+|c_2|>0$. Then the following conditions are equivalent:
\begin{enumerate}
	\item [(a)] $J_g$ has closed range on $F^p_{\alpha,w}$;
	\item [(b)] $J_g$ is bounded below on $F^p_{\alpha,w}$;
	\item [(c)] $c_2\neq0$.
\end{enumerate}
\end{theorem}
\begin{proof}
Since $g$ is not a constant, $J_g$ is one-to-one on $\h(\C)$. The equivalence between (a) and (b) is a consequence of the closed graph theorem; see for instance \cite[Theorem 3.2]{An11}. We now establish the equivalence between (b) and (c). By Lemma \ref{hat-norm}, $J_g$ is bounded below on $F^p_{\alpha,w}$ if and only if there exists $C>0$ such that for any $f\in F^p_{\alpha,w}=F^p_{\alpha,\widehat{w}}$,
$$\|J_gf\|_{F^p_{\alpha,\widehat{w}}}^p\geq C\|f\|_{F^p_{\alpha,\widehat{w}}}^p,$$
which, due to \cite[Theorem 1.1]{CFP}, is equivalent to that there exists $C'>0$ such that
$$\int_{\C}|f(z)|^p\left(\frac{|2c_2z+c_1|}{1+|z|}\right)^pe^{-\frac{p\alpha}{2}|z|^2}\widehat{w}(z)dA(z)
\geq C'\|f\|_{F^p_{\alpha,\widehat{w}}}^p.$$
For $\epsilon>0$, write $G_{\epsilon}:=\{z\in\C:|2c_2z+c_1|\geq \epsilon(1+|z|)\}$. Then by Corollary \ref{multiplier}, $J_g$ is bounded below on $F^p_{\alpha,w}$ if and only if there exist $R,\epsilon>0$ and $\delta\in(0,1)$ such that \eqref{G-t} holds. If $c_2=0$, then it is easy to verify that for any $R,\epsilon>0$,
$$G_{\epsilon}\cap D(a,R)=\emptyset$$
whenever $|a|>R+\frac{|c_1|}{\epsilon}$. Consequently, \eqref{G-t} fails for any $R,\epsilon>0$ and $\delta\in(0,1)$. If $c_2\neq0$, choose $\epsilon=|c_2|$ and $R=3\left(1+|c_1|/|c_2|\right)$. Then, noting that for any $z\in\C$ with $|z|\geq1+|c_1|/|c_2|$,
$$|2c_2z+c_1|\geq2|c_2||z|-|c_1|\geq|c_2|(1+|z|),$$
we have
$$\C\setminus D\left(0,1+\frac{|c_1|}{|c_2|}\right)\subset G_{\epsilon}.$$
For any $a\in\C$, write
\begin{equation*}
	\tilde{a}=
	\begin{cases}
		a+2\left(1+\frac{|c_1|}{|c_2|}\right)\frac{a}{|a|},&   \text{$a\neq0$},\\
		2\left(1+\frac{|c_1|}{|c_2|}\right),&   \text{$a=0$}.
	\end{cases}
\end{equation*}
Then for $\xi\in D(\tilde{a},1+|c_1|/|c_2|)$,
$$|\xi-a|\leq|\xi-\tilde{a}|+|\tilde{a}-a|<3\left(1+\frac{|c_1|}{|c_2|}\right),$$
and
$$|\xi|\geq|\tilde{a}|-|\xi-\tilde{a}|>1+\frac{|c_1|}{|c_2|}.$$
That is,
$$D\left(\tilde{a},1+\frac{|c_1|}{|c_2|}\right)\subset D(a,R)\setminus D\left(0,1+\frac{|c_1|}{|c_2|}\right)
\subset G_{\epsilon}\cap D(a,R).$$
Therefore,
$$\widehat{w}\big(G_{\epsilon}\cap D(a,R)\big)\geq \widehat{w}\left(D\left(\tilde{a},1+\frac{|c_1|}{|c_2|}\right)\right),$$
which, combined with Lemma \ref{w-esti}, implies that there exists $C_{\ref{w-esti}}=C_{\ref{w-esti}}(R,w)>0$ such that
\begin{align*}
\widehat{w}\big(G_{\epsilon}\cap D(a,R)\big)&\geq\int_{D\left(\tilde{a},1+\frac{|c_1|}{|c_2|}\right)}\widehat{w}(\xi)dA(\xi)\\
&\geq\frac{1}{C_{\ref{w-esti}}}\left(1+\frac{|c_1|}{|c_2|}\right)^2w\big(D(a,R)\big)\\
&=\frac{1}{9C_{\ref{w-esti}}}\int_{D(a,R)}\frac{w\big(D(a,R)\big)}{\pi}dA(\xi)\\
&\geq\frac{1}{9C_{\ref{w-esti}}^2}\widehat{w}\big(D(a,R)\big).
\end{align*}
Hence \eqref{G-t} holds and the proof is complete.
\end{proof}

\subsection{Atomic decomposition of $F^p_{\alpha,w}$}

In this subsection, we use Theorem \ref{main2} to establish an atomic decomposition for the weighted Fock spaces $F^p_{\alpha,w}$ with $1<p<\infty$ and $w\in A^{\res}_p$. We will need the following lemma, which actually indicates that the lattice $r\Z^2$ is a sampling sequence for $F^p_{\alpha,w}$ whenever $r>0$ is small enough.

\begin{lemma}\label{sampling}
Let $0<p,\alpha<\infty$ and $w\in A_{\infty}^{\res}$. Then there exists $r_{\ref{sampling}}=r_{\ref{sampling}}(p,\alpha,w)>0$ satisfying that for each $r\in(0,r_{\ref{sampling}})$, there exists a constant $C_{\ref{sampling}}=C_{\ref{sampling}}(p,\alpha,r,w)\geq1$ such that
$$C_{\ref{sampling}}^{-1}\|f\|_{F^p_{\alpha,w}}^p
\leq\sum_{\nu\in r\mathbb{Z}^2}|f(\nu)|^pe^{-\frac{p\alpha}{2}|\nu|^2}w\big(D(\nu,1)\big)
\leq C_{\ref{sampling}}\|f\|_{F^p_{\alpha,w}}^p$$
holds for any $f\in F^p_{\alpha,w}$.
\end{lemma}
\begin{proof}
By \cite[Lemma 3.1]{CFP}, we know that there exists $C=C(p,\alpha,w)>0$ such that for any $f\in F^p_{\alpha,w}$ and any $r>0$,
\begin{align*}
\sum_{\nu\in r\mathbb{Z}^2}|f(\nu)|^pe^{-\frac{p\alpha}{2}|\nu|^2}w\big(D(\nu,1)\big)
&\leq C\sum_{\nu\in r\Z^2}\int_{D(\nu,1)}|f(z)|^pe^{-\frac{p\alpha}{2}|z|^2}w(z)dA(z)\\
&\leq C\left(1+\frac{2}{r}\right)^2\|f\|_{F^p_{\alpha,w}}^p.
\end{align*}
We next concentrate on the first inequality. Let $C_{\ref{w-esti}}$ be the constant from Lemma \ref{w-esti} associated to $R=1$ and $w$, and let $r_{\ref{main2}}$ be the radius from Theorem \ref{main2} associated to $p,\alpha,R=1,w$, $\epsilon=4/9C_{\ref{w-esti}}^2$ and $\delta=49/400C_{\ref{w-esti}}^2$. Write $r_{\ref{sampling}}=\min\{2r_{\ref{main2}},1/5\}$. We are going to show that each $r\in(0,r_{\ref{sampling}})$ has the desired property. To this end, fix $r\in\big(0,r_{\ref{sampling}}\big)$, and define
$$G:=\bigcup_{\nu\in r\Z^2}D(\nu,r/2),\quad \mu:=\sum_{\nu\in r\Z^2}w\big(D(\nu,1)\big)\delta_{\nu},$$
where $\delta_{\nu}$ is the Dirac point mass at $\nu$. For any $a\in\C$, since $r<1/5$, we have $D(\nu,r/2)\subset D(a,1)$ whenever $|\nu-a|<9/10$. Consequently,
\begin{equation}\label{subset}
\bigcup_{\nu\in r\Z^2\cap D(a,9/10)}D(\nu,r/2)\subset G\cap D(a,1).
\end{equation}
Using Lemma \ref{w-esti}, we obtain that for any $\nu\in D(a,9/10)$,
\begin{equation}\label{w-geq}
\widehat{w}\big(D(\nu,r/2)\big)\geq\frac{r^2}{4C_{\ref{w-esti}}^2}\widehat{w}\big(D(a,1)\big).
\end{equation}
On the other hand, there exists $\nu_a\in r\Z^2$ such that $a\in Q_r(\nu_a)$, which implies
$$Q_{7/10}(\nu_a)\subset D(\nu_a,7/10)\subset D(a,9/10),$$
and so
\begin{equation}\label{number-geq}
|r\Z^2\cap D(a,9/10)|\geq|r\Z^2\cap Q_{7/10}(\nu_a)|\geq\frac{49}{100r^2}.
\end{equation}
Putting \eqref{subset}, \eqref{w-geq} and \eqref{number-geq} together, we obtain that for $a\in\C$,
\begin{align*}
\widehat{w}\big(G\cap D(a,1)\big)
&\geq\sum_{\nu\in r\Z^2\cap D(a,9/10)}\widehat{w}\big(D(\nu,r/2)\big)\\
&\geq\frac{r^2}{4C_{\ref{w-esti}}^2}\widehat{w}\big(D(a,1)\big)\cdot|r\Z^2\cap D(a,9/10)|\\
&\geq\frac{49}{400C_{\ref{w-esti}}^2}\widehat{w}\big(D(a,1)\big).
\end{align*}
We next claim that $\|\widehat{\mu}_w\|_{L^{\infty}}<\infty$ and
\begin{equation}\label{claim}
G=\left\{z\in\C:\frac{\mu\big(D(z,r/2)\big)}{\widehat{w}\big(D(z,r/2)\big)}
\geq\frac{4}{9C_{\ref{w-esti}}^2}\|\widehat{\mu}_w\|_{L^{\infty}}\right\}.
\end{equation}
In fact, Lemma \ref{w-esti} gives
$$\|\widehat{\mu}_w\|_{L^{\infty}}=\sup_{z\in\C}\frac{\mu\big(D(z,1)\big)}{w\big(D(z,1)\big)}
=\sup_{z\in\C}\frac{\sum_{\nu\in r\Z^2\cap D(z,1)}w\big(D(\nu,1)\big)}{w\big(D(z,1)\big)}\leq C_{\ref{w-esti}}\left(1+\frac{2}{r}\right)^2.$$
Moreover, it is clear that for any $z\notin G$, $\mu\big(D(z,r/2)\big)=0$. If $z\in G$, then there exists a unique $\nu_z\in r\Z^2$ such that $z\in D(\nu_z,r/2)$, and so Lemma \ref{w-esti} together with the fact that $r<1$ yields
\begin{align*}
\frac{\mu\big(D(z,r/2)\big)}{\widehat{w}\big(D(z,r/2)\big)}
=\frac{w\big(D(\nu_z,1)\big)}{\widehat{w}\big(D(z,r/2)\big)}
\geq\frac{4}{r^2C_{\ref{w-esti}}}
\geq\frac{4\|\widehat{\mu}_w\|_{L^{\infty}}}{C_{\ref{w-esti}}^2(2+r)^2}
>\frac{4}{9C_{\ref{w-esti}}^2}\|\widehat{\mu}_w\|_{L^{\infty}}.
\end{align*}
Therefore, \eqref{claim} holds. Since $r/2<r_{\ref{main2}}$, we deduce from Theorem \ref{main2} that $\mu$ is a reverse Carleson measure for $F^p_{\alpha,w}$. Consequently, there exists $C=C(p,\alpha,r,w)>0$ such that for any $f\in F^p_{\alpha,w}$,
$$\|f\|_{F^p_{\alpha,w}}^p\leq C\int_{\C}|f(z)|^pe^{-\frac{p\alpha}{2}|z|^2}d\mu(z)
=C\sum_{\nu\in r\Z^2}|f(\nu)|^pe^{-\frac{p\alpha}{2}|\nu|^2}w\big(D(\nu,1)\big).$$
The proof is complete.
%
%
\end{proof}

Based on Lemma \ref{sampling}, we can establish the following atomic decomposition for $F^p_{\alpha,w}$.

\begin{theorem}\label{atomic}
Let $1<p<\infty$, $0<\alpha<\infty$ and $w\in A_p^{\res}$. Then there exists $r_0>0$ such that for any $r\in(0,r_0)$, the space $F^p_{\alpha,w}$ consists exactly of the following functions:
\begin{equation}\label{decom}
f(z)=\sum_{\nu\in r\Z^2}c_{\nu}\frac{K_{\nu}(z)}{\|K_{\nu}\|_{F^p_{\alpha,w}}},\quad z\in\C,
\end{equation}
where $\{c_{\nu}\}_{\nu\in r\Z^2}\in l^p(r\Z^2)$. Moreover, there exists $C=C(p,\alpha,r,w)\geq1$ such that for all $f\in F^p_{\alpha,w}$,
$$C^{-1}\|f\|_{F^p_{\alpha,w}}\leq \inf\|\{c_{\nu}\}\|_{l^p(r\Z^2)}\leq C\|f\|_{F^p_{\alpha,w}},$$
where the infimum is taken over all sequences $\{c_{\nu}\}$ that give rise to the decomposition \eqref{decom}.
\end{theorem}

Before proceeding, we recall some preliminary results. Let $1<p<\infty$, $0<\alpha<\infty$ and $w\in A_p^{\res}$, and define the integral pairing $\langle\cdot,\cdot\rangle_{\alpha}$ by
$$\langle f,g\rangle_{\alpha}:=\frac{\alpha}{\pi}\int_{\C}f(z)\overline{g(z)}e^{-\alpha|z|^2}dA(z).$$
It was proved in \cite[Lemma 3.1]{CHW} that the dual space of $F^p_{\alpha,w}$ can be identified with $F^{p'}_{\alpha,w'}$ under the pairing $\langle \cdot,\cdot\rangle_{\alpha}$, where $w':=w^{-p'/p}$ is the dual weight of $w$. Moreover, $w'\in A_{p'}^{\res}$, and H\"older's inequality ensures that
\begin{equation}\label{w-w'}
w\big(D(z,1)\big)^{1/p}\asymp w'\big(D(z,1)\big)^{-1/p'},\quad z\in\C.
\end{equation}
Here and in the sequel, we write $A\asymp B$ to denote that $C^{-1}B\leq A\leq CB$ for some inessential constant $C>0$. Using \cite[Proposition 3.3 and Lemma 3.5]{CFP}, one can obtain that
\begin{equation}\label{g=pg}
g(z)=P_{\alpha}g(z)=\langle g,K_z\rangle_{\alpha},\quad z\in\C,\quad g\in F^{p'}_{\alpha,w'}.
\end{equation}
We are now ready to prove Theorem \ref{atomic}.

\begin{proof}[Proof of Theorem \ref{atomic}]
For any $r>0$, let the operator $S$ be defined for $\{c_{\nu}\}\in l^p(r\Z^2)$ by
$$S(\{c_{\nu}\}):=\sum_{\nu\in r\Z^2}c_{\nu}\frac{K_{\nu}}{\|K_{\nu}\|_{F^p_{\alpha,w}}}.$$
Due to \cite[Proposition 4.2]{CFP}, $S$ is bounded from $l^p(r\Z^2)$ into $F^p_{\alpha,w}$. We now determine the adjoint of $S$. For any $\{c_{\nu}\}_{r\Z^2}\in l^p(r\Z^2)$ and $g\in F^{p'}_{\alpha,w'}$, Fubini's theorem together with \eqref{g=pg} yields that
\begin{align*}
\langle\{c_{\nu}\},S^*g\rangle_{l^2}
&=\langle S(\{c_{\nu}\}),g\rangle_{\alpha}\\
&=\frac{\alpha}{\pi}\int_{\C}\left(\sum_{\nu\in r\Z^2}c_{\nu}\frac{K_{\nu}(z)}{\|K_{\nu}\|_{F^p_{\alpha,w}}}\right)
    \overline{g(z)}e^{-\alpha|z|^2}dA(z)\\
&=\sum_{\nu\in r\Z^2}c_{\nu}\|K_{\nu}\|_{F^p_{\alpha,w}}^{-1}
    \frac{\alpha}{\pi}\int_{\C}\overline{g(z)}K_{\nu}(z)e^{-\alpha|z|^2}dA(z)\\
&=\sum_{\nu\in r\Z^2}c_{\nu}\overline{g(\nu)}\|K_{\nu}\|_{F^p_{\alpha,w}}^{-1}.
\end{align*}
Therefore, the adjoint of $S$ is given by
$$S^*g=\left\{g(\nu)\|K_{\nu}\|_{F^p_{\alpha,w}}^{-1}\right\}_{\nu\in r\Z^2},\quad g\in F^{p'}_{\alpha,w'}.$$
Let $r_0$ be the radius from Lemma \ref{sampling} associated with $p'$, $\alpha$ and $w'$. Then for $r\in(0,r_0)$, we combine Lemmas \ref{test}, \ref{sampling} with the estimate \eqref{w-w'} to obtain that for any $g\in F^{p'}_{\alpha,w'}$,
\begin{align*}
\|S^*g\|_{l^{p'}(r\Z^2)}^{p'}
&=\sum_{\nu\in r\Z^2}|g(\nu)|^{p'}\|K_{\nu}\|^{-p'}_{F^p_{\alpha,w}}\\
&\asymp\sum_{\nu\in r\Z^2}|g(\nu)|^{p'}e^{-\frac{p'\alpha}{2}|\nu|^2}w\big(D(\nu,1)\big)^{-p'/p}\\
&\asymp\sum_{\nu\in r\Z^2}|g(\nu)|^{p'}e^{-\frac{p'\alpha}{2}|\nu|^2}w'\big(D(\nu,1)\big)\\
&\asymp\|g\|^{p'}_{F^{p'}_{\alpha,w'}},
\end{align*}
where the implicit constants are independent of $g$. Therefore, $S^*:F^{p'}_{\alpha,w'}\to l^{p'}(r\Z^2)$ is an isomorphic embedding onto a closed subspace of $l^{p'}(r\Z^2)$, which implies that $S:l^p(r\Z^2)\to F^p_{\alpha,w}$ is onto. Consequently, for any $f\in F^p_{\alpha,w}$, there exists some $\{c_{\nu}\}\in l^p(r\Z^2)$ such that
$$f(z)=S(\{c_{\nu}\})(z)=\sum_{\nu\in r\Z^2}c_{\nu}\frac{K_{\nu}(z)}{\|K_{\nu}\|_{F^p_{\alpha,w}}},\quad z\in\C.$$
It remains to establish the norm estimate. Let $\mathrm{Ker}\,S$ be the kernel of $S$, and let the operator $\tilde{S}$ be defined on the quotient space $l^p(r\Z^2)/\mathrm{Ker}\,S$ by
$$\tilde{S}(\{c_{\nu}\}+\mathrm{Ker}\,S):=S(\{c_{\nu}\}).$$
Then $\tilde{S}$ is an isomorphism from $l^p(r\Z^2)/\mathrm{Ker}\,S$ onto $F^p_{\alpha,w}$. Hence for any $f\in F^p_{\alpha,w}$,
$$\|f\|_{F^p_{\alpha,w}}\asymp\|\tilde{S}^{-1}f\|_{l^p(r\Z^2)/\mathrm{Ker}\,S}=\inf\|\{c_{\nu}\}\|_{l^p(r\Z^2)},$$
where the implicit constant is independent of $f$, and the infimum is taken over all sequences $\{c_{\nu}\}$ with $f=S(\{c_{\nu}\})$. The proof is complete.
\end{proof}

\medskip





\end{document}